\documentclass[12pt,reqno]{amsart}
\usepackage{amssymb,mathrsfs,amsmath,amsfonts,amsthm,graphicx,pdfpages,MnSymbol,relsize,quiver}
\usepackage{xcolor}
\usepackage{tikz, tikz-cd}
\usepackage[section]{algorithm}
\usepackage{algorithmic}
\usepackage{wasysym}

\newcommand{\define}[1]{{\bf \boldmath{#1}}}

\newcommand{\namedset}[1]{\mathbb{#1}}
\newcommand{\N}{\namedset N}

\newcommand{\namedcat}[1]{\mathsf{#1}}

\newcommand{\Cat}{\namedcat{Cat}}
\newcommand{\CMC}{\namedcat{CMC}}
\newcommand{\CMon}{\namedcat{CMon}}

\newcommand{\Set}{\namedcat{Set}}
\newcommand{\Petri}{\namedcat{Petri}}
\newcommand{\cPetri}{\namedcat{cPetri}}

\newcommand{\UNFOLD}{\mathtt{UNFOLD}}

\newcommand{\Ob}{\:\mathrm{Ob}}
\newcommand{\Mor}{\:\mathrm{Mor}}

\newcommand{\Span}{\mathsf{Span}}

\newcommand{\U}{\mathbb{U}}
\newcommand{\Dbl}{\mathsf{Dbl}}

\newcommand{\Disp}{\mathsf{Disp}}

\makeatletter
\newcommand*{\relrelbarsep}{.386ex}
\newcommand*{\relrelbar}{%
  \mathrel{%
    \mathpalette\@relrelbar\relrelbarsep
  }%
}
\newcommand*{\@relrelbar}[2]{%
  \raise#2\hbox to 0pt{$\m@th#1\relbar$\hss}%
  \lower#2\hbox{$\m@th#1\relbar$}%
}
\providecommand*{\rightrightarrowsfill@}{%
  \arrowfill@\relrelbar\relrelbar\rightrightarrows
}
\providecommand*{\leftleftarrowsfill@}{%
  \arrowfill@\leftleftarrows\relrelbar\relrelbar
}
\providecommand*{\xrightrightarrows}[2][]{%
  \ext@arrow 0359\rightrightarrowsfill@{#1}{#2}%
}
\providecommand*{\xleftleftarrows}[2][]{%
  \ext@arrow 3095\leftleftarrowsfill@{#1}{#2}%
}
\makeatother

\usepackage{subfiles}
\usepackage{stackengine}
\usepackage{mathtools}
\usepackage[utf8]{inputenc}
\usepackage{color}
\usepackage[a4paper,top=3cm,bottom=3cm,inner=3cm,outer=3cm]{geometry}

\usepackage{comment}

\usepackage{graphicx}
\usepackage{adjustbox}
\usepackage[all,2cell]{xy}\UseAllTwocells\SilentMatrices
\definecolor{darkgreen}{rgb}{0,0.45,0}
\usepackage[colorlinks,citecolor=darkgreen,urlcolor=gray,final,hyperindex,linktoc=page,pagebackref]{hyperref}

\usepackage[capitalize]{cleveref}
\crefname{equation}{}{}
\crefname{item}{}{}
\usepackage{enumerate}

\newtheorem*{thm*}{Theorem}
\theoremstyle{remark}
\newtheorem*{rmk*}{Remark}
\newtheorem*{lem*}{Lemma}
\theoremstyle{plain}
\newtheorem*{defn*}{Definition}
\newtheorem*{cor*}{Corollary}
\theoremstyle{definition}
\newtheorem*{examples*}{Examples}
\newtheorem{prop*}{Proposition}

\theoremstyle{plain}
\newtheorem{thm}{Theorem}[section]
\theoremstyle{plain}
\newtheorem{prop}[thm]{Proposition}
\theoremstyle{remark}

\theoremstyle{plain}

\theoremstyle{plain}

\theoremstyle{definition}
\newtheorem{defn}[thm]{Definition}
\theoremstyle{definition}
\newtheorem{examples}[thm]{Example}

\newcommand{\maps}{\colon}

\usetikzlibrary{arrows,positioning,fit,matrix,shapes.geometric,external}

\newcommand*\pgfdeclareanchoralias[3]{%
  \expandafter\def\csname pgf@anchor@#1@#3\expandafter\endcsname
     \expandafter{\csname pgf@anchor@#1@#2\endcsname}}

\tikzset{
    circnode/.style={
      circle, draw=red, very thin, outer sep=0.025em, minimum size=2em,
      fill=red, text centered},
    integral/.style={
      circle, draw=black, very thick, outer sep=0.025em,
      minimum size=2em, fill=blue!5, text centered},
    multiply/.style={
      circle, draw=black, very thick, outer sep=0.025em,
      minimum size=2em, fill=blue!5, text centered},
    zero/.style={
      circle, draw=black, very thick, minimum size=0.15cm, fill=black,
      inner sep=0, outer sep=0},
    bang/.style={
      circle, draw=black, very thick, minimum size=0.15cm, fill=green!10,
      inner sep=0, outer sep=0},
    delta/.style={
      regular polygon, regular polygon sides=3, minimum size=0.4cm, inner
      sep=0, outer sep=0.025em, draw=black, very thick, fill=green!10},
    codelta/.style={
      regular polygon, regular polygon sides=3, shape border rotate=180, minimum size=0.4cm,
      inner sep=0, outer sep=0.025em, draw=black, very thick, fill=green!10},
    plus/.style={
      regular polygon, regular polygon sides=3, shape border rotate=180, minimum size=0.4cm,
      inner sep = 0, outer sep=0.025em, draw=black, very thick, fill=black},
    coplus/.style={
      regular polygon, regular polygon sides=3, minimum size=0.4cm,
      inner sep = 0, outer sep=0.025em, draw=black, very thick, fill=black},
    sqnode/.style={
      regular polygon, regular polygon sides=4, minimum size=2.6em,
      draw=black, very thick, inner sep=0.2em, outer sep=0.025em,
      fill=yellow!10, text centered},
    bigcirc/.style={
      circle, draw=black, very thick, text width=1.6em, outer sep=0.025em,
      minimum height=1.6em, fill=blue!5, text centered}
}

\tikzstyle{tri}=[regular polygon,regular polygon sides=3,shape border rotate=1
80,fill=none,draw=black]

\pgfdeclareanchoralias{regular polygon}{corner 2}{left copy}
\pgfdeclareanchoralias{regular polygon}{corner 3}{right copy}
\pgfdeclareanchoralias{regular polygon}{corner 3}{addend}
\pgfdeclareanchoralias{regular polygon}{corner 2}{summand}

\usepackage{tikz}
\usepackage{tikz-cd}
\usetikzlibrary{backgrounds,circuits,circuits.ee.IEC,shapes,fit,matrix}

\tikzstyle{simple}=[-,line width=2.000]
\tikzstyle{arrow}=[-,postaction={decorate},decoration={markings,mark=at position .5 with {\arrow{>}}},line width=1.100]
\pgfdeclarelayer{edgelayer}
\pgfdeclarelayer{nodelayer}
\pgfsetlayers{edgelayer,nodelayer,main}

\tikzstyle{none}=[inner sep=0pt]

\definecolor{lblue}{rgb}{0,250,255}
\tikzstyle{species}=[circle,fill=yellow,draw=black,scale=1.15]
\tikzstyle{transition}=[rectangle,fill=lblue,draw=black,scale=1.15]
\tikzstyle{inarrow}=[->, >=stealth, shorten >=.03cm,line width=1.5]
\tikzstyle{empty}=[circle,fill=none, draw=none]
\tikzstyle{inputdot}=[circle,fill=black,draw=black, scale=.25]
\tikzstyle{inputarrow}=[->,draw=purple, shorten >=.05cm]
\tikzstyle{simple}=[-,draw=black,line width=1.000]

\usetikzlibrary{arrows,shapes,decorations,automata,backgrounds,petri}
\tikzstyle{place}=[circle,thick,draw=blue!75,fill=blue!20,minimum size=6mm]
\tikzstyle{red place}=[place,draw=red!75,fill=red!20]
\tikzstyle{transition}=[rectangle,thick,draw=black!75,
  			  fill=black!20,minimum size=4mm]

\newcommand{\Int}{\textstyle{\int}}

\newcommand{\Kl}{\mathsf{Kl}}

\title[Colored Petri Nets are Functors]{Colored Petri Nets are Monoidal Double Functors}

\author[Master and Moeller]{Jade Master and Joe Moeller}

\email{jmoeller31415@gmail.com, jadeedenstarmaster@gmail.com}

\begin{document}

\begin{abstract}
We give a characterization of colored Petri nets as monoidal double functors. Framing colored Petri nets in terms of category theory allows for canonical definitions of various well-known constructions on colored Petri nets. In particular, we show how morphisms of colored Petri nets may be understood as natural transformations. The displayed category construction explains how lax double functors are equivalent to functors with codomain their former domain. We use this result to characterize the unfolding of colored Petri nets in terms of free symmetric monoidal categories.
\end{abstract}

\maketitle
\setcounter{tocdepth}{1} 
\tableofcontents{}

Petri nets are a widely studied formalism for representing networks of processes and resources. When a system has sufficient complexity, modeling with ordinary Petri nets quickly becomes unwieldy. Colored Petri nets alleviate this problem by allowing repetitive structures to be encoded as extra data attached to the Petri net. This extra data can be thought of as a more detailed interpretation of the places and transitions in the underlying Petri net. From the ordinary definition of colored Petri net \cite{jensen2013coloured}, it is not clear how a colored Petri net can be obtained by applying a semantic interpretation to an ordinary Petri net. In this paper, we make precise the way in which colored Petri nets are syntax applied to semantics. Indeed, in Proposition \ref{correspond}, we show how a colored Petri net naturally determines a symmetric monoidal double functor
\[ K \maps FP \to \Span(\CMon). \]
Here $P$ is an ordinary Petri net, 
$FP$ is the categorical semantics of $P$ upgraded trivially to a double category, and $\Span(\CMon)$ is the double category whose loose morphisms are spans of commutative monoids. Here $FP$ is regarded as the ``syntax'', and $\Span(\CMon)$ is the ``semantics''.

Rephrasing colored Petri nets in this way has many advantages. In particular,
\begin{itemize}
    \item Transformations between double functors give a natural definition of morphism between colored Petri nets (Definition \ref{morphism}). In this paper, we compare this definition to the definition of morphism of colored Petri nets to the definition given by Lakos \cite{lakosabstraction, lakoscomposing}. We find that our more abstract definition retains many of the desirable qualities of Lakos' definition.
 \item In 1990 \cite{monoids} Montanari and Meseguer elegantly illustrated the connection between Petri nets and monoidal category theory by giving an operational semantics functor which turns each Petri net into a free symmetric monoidal category. The presentation of colored Petri nets as functors allows for an analogous operational semantics to be defined based on the Grothendieck construction. In Definition \ref{semantics} we construct this operational semantics category and in Proposition \ref{soundness} we prove that it faithfully captures the firing sequences of a colored Petri net.
\item Unfolding is an algorithmic process which turns a colored Petri net into an ordinary one for the purposes of analysis \cite{jensen2013coloured}. In this paper we describe unfolding as a functor 
    \[ \mathtt{UNFOLD} \maps \cPetri \to \Petri.\]
    The correctness of unfolding can be expressed concisely using category theory. In Theorem \ref{unfold}, we show that there is a diagram
\[
\begin{tikzcd} 
& \Petri \ar[dr,"F"] & \\
\cPetri \ar[ur,"\UNFOLD"] \ar[rr,"\int",swap]  & &\CMC.
\end{tikzcd}
\]
which commutes up to isomorphism. This commutativity expresses the fact that unfolding preserves processes by providing an isomorphism between the sequences of processes of a colored Petri net and the sequences of processes of its unfolding.
\end{itemize}

A plan of this paper is as follows: 
\begin{itemize}
    \item In Section \ref{sec:petri} we give a review of Petri nets and their process semantics from the perspective of category theory.
    \item In Section \ref{sec:coloredpetri} we introduce a definition of colored Petri nets and their morphisms and discuss their relationship to the traditional notions of colored Petri nets and their morphisms.
   \item In Section \ref{sec:unfolding} we use the displayed category construction to give a categorical operational semantics of colored Petri nets.
   \item In Section \ref{unfoldingtwo} we show that the categorical operational semantics of a colored Petri net is freely generated by its unfolding.
\end{itemize}
\subsection{Related Work}
A characterization of colored Petri nets as functors has already appeared in \cite{Guarded}.  Although \cite{Guarded} and the current paper use similar ideas, they were developed independently. Furthermore, the colored Petri nets studied in this paper have some differences from the colored Petri nets studied by Genovese and Spivak:
\begin{itemize}
\item The domain is not the same: The free symmetric monoidal category used by Genovese and Spivak is the non-commutative variant developed by Sassone \cite{SassoneStrong}. We use the free commutative monoidal category functor originated in \cite{monoids} and clarified in \cite{GeneralizedPetriNets}. Our approach gives an operational semantics which highlights the \emph{collective token philosophy}; the order and identities of the tokens are not represented in this operational semantics. On the other hand, the construction of Sassone follows the \emph{individual token philosophy} which does keep track of this data. We chose the free commutative monoidal category because it's existence follows from more general categorical principles. To freely generate symmetric monoidal categories which exhibit the individual token philosophy, we suggest using pre-nets, a non-commutative variant of Petri nets which have a more straightforward adjunction into the right sort of symmetric monoidal category \cite{functorialsemantics}.
\item The arc inscriptions are not the same: In Genovese and Spivak's approach, each transition is mapped to an isomorphism class of spans of sets. On the other hand, in our approach each transition is annotated with a span in the Kleisli category $\Kl(\mathbb{N})$. Using the Kleisli category allows the colors of the transitions to have multisets of inputs and outputs whereas the colors on the transitions in \cite{Guarded} may only map unique inputs to unique outputs. As proven in Proposition \ref{correspond}, our definition matches the original definition in \cite{jensen}.
\item The definition in \cite{Guarded} is evil: 
Evil is not meant as a moral judgement, just that the model does not satisfy the principle of equivalence (i.e.\ \cite{evil}).
The codomain of a colored Petri net in Genovese and Spivak's approach is a decategorification of the double category $\Span(\Set)$ whose morphisms are equivalence classes of spans. Defining a functor into this category requires choosing a non-canonical representatives from these equivalence classes. If these representatives are chosen inconsistently, two colored Petri nets which are equal may have totally different presentations. To make matters worse, in \cite{Guarded} the authors have also decategorified the cartesian monoidal structure of $\Span(\Set)$ to make it strictly associative and unital. With so many equivalence relations imposed on spans of sets, it is unclear how to choose representatives in consistent way. In this paper we have remedied this issue by using lax functors and double categories so that each transition is annotated by an actual span rather than an equivalence class.

\end{itemize}
In \cite{baez2019network}, the authors also use the Grothendieck construction to attach extra information to a Petri net. In particular, they use the monoidal Grothendieck construction \cite{moeller2020monoidal} to understand network models with catalysts. 

\section{Petri Nets and Their Semantics}
\label{sec:petri}

One one hand, Petri nets can be thought of as (directed, multi-) graphs with two distinct types of nodes. However, it is more in line with the spirit  of Petri nets to think of them as graphs in which arrows are permitted to have a multiset of source nodes, and a multiset of target nodes. 
\begin{defn}
A commutative monoid is a set equipped with an associative, unital, and commutative, binary operation. A commutative monoid homomorphism is a function between commutative monoids which preserves the binary operation and the identity. Let $\CMon$ denote the category of commutative monoids and their homomorphisms. The free commutative monoid on a set $X$ is given by 
\[
    \N[X]  = \{ a \maps X \to \N \, | \, f(x) \neq 0 \text{ for only finitely many elements of } X \} 
\]
with addition given pointwise. Let $\N$ denote the free commutative monoid monad $\N \maps \Set \to \Set$.
\end{defn} 



\begin{defn}
\label{def:Petri}
    A \define{Petri net} is a pair of functions of the following form.
    \[\begin{tikzcd}
        T 
        \arrow[r, shift left, "s"]
        \arrow[r, shift right, "t", swap]
        & 
        \N[S]
    \end{tikzcd}\]
    We call $T$ the set of \define{transitions}, $S$ the set of \define{places}, $s$ the \define{source} function and $t$ the \define{target} function. A \define{Petri net morphism} $(f,g) \maps (T,S,t,s) \to (T',S',t',s')$ consists of a pair of functions $f \maps T \to T'$ and $g \maps S \to S'$ such that the following diagrams commute.
    \[
    \begin{tikzcd}
        T
        \arrow[r, "s"]
        \arrow[d, "f", swap]
        &
        \N[S]
        \arrow[d, "{\N[g]}"]
        \\
        T'
        \arrow[r, "s'", swap]
        &
        \N[S']
    \end{tikzcd}
    \qquad
    \begin{tikzcd}
        T
        \arrow[r, "t"]
        \arrow[d, "f", swap]
        &
        \N[S]
        \arrow[d, "{\N[g]}"]
        \\
        T'
        \arrow[r, "t'", swap]
        &
        \N[S']
    \end{tikzcd}\] 
    Let $\Petri$ denote the category of Petri nets and Petri net morphisms, with composition defined pointwise.
\end{defn}

\noindent Our definition of morphisms between Petri nets follows that found in the work of Baez and Master \cite{GeneralizedPetriNets, OpenPetriNets}. This differs from the earlier definition used by Sassone \cite{SassoneStrong, SassoneAxiom} and Degano--Meseguer--Montanari \cite{DMM}, in that the definition presented here requires that the homomorphism between free commutative monoids come from a function between the sets of places. We have chosen our definitions because they are more well-behaved categorically. In particular, our category $\Petri$ is cocomplete (\cite{GeneralizedPetriNets}, Prop.\ 3.10).

Petri nets have a natural semantics which is described by \emph{the token game}. Each place of a Petri net is equipped with a natural number of tokens.  Such an assignment is called a \define{marking}, and is formally represented by an element $m \in \N[S]$, or equivalently, a function $m \maps S \to \N$. Players are then allowed to shuffle the tokens around by \emph{firing} transitions. Given a particular marking $m$, a transition $\tau$ can be fired if there are enough tokens in all of its source places. Specifically, $\tau$ can be fired if and only if $s(\tau) \leq m$ under the pointwise order on $\N[S]$. Formally, an \define{atomic firing} of $P$ is a tuple $(\tau,x,y)$, where $\tau$ is a transition and $x$ and $y$ are markings such that $x-s(\tau) + t(\tau) = y$. Evocatively, we write such tuples as $\tau \maps x \to y$. Firings can be built up from atomic firings via parallel and sequential composition. Given two firings $\tau \maps x \to y$ and $\tau' \maps x' \to y'$, we write their parallel firing as $\tau + \tau'  \maps x + x' \to y + y'$, where the sum of the markings is the sum in $\N[S]$. Given two firings $\tau \maps x \to y$ and $\sigma \maps y \to z$, we write their sequential firing as $\sigma \circ \tau \maps x \to z$. 

As suggested by our choice of notation, these operations fit together into the structure of a monoidal category, which we will denote as $FP$. The objects of $FP$ are markings of $P$, the morphisms are ``firing patterns'', or legal sequences of transition firings, potentially in parallel, up to a very natural equivalence of patterns. Meseguer and Montanari were the first to demonstrate the relationship between Petri nets and symmetric monoidal categories \cite{monoids}.
Petri nets generate a specific kind of symmetric monoidal category.
\begin{defn}
A \define{commutative monoidal category} is a commutative monoid object internal to the 1-category $\Cat$. Explicitly, a commutative monoidal category is a strict monoidal category $(C,\otimes,I)$, such that each symmetry map $\sigma_{a,b} \colon a \otimes b \to b \otimes a$ is an identity map.
\end{defn}
\noindent A commutative monoidal category is precisely a category where the objects and morphisms form commutative monoids and the structure maps are commutative monoid homomorphisms. 
\begin{defn}
	Let $\CMC$ be the category whose objects are commutative monoidal categories and whose morphisms are strict monoidal functors.
\end{defn}
\noindent Note that every monoidal functor between commutative monoidal categories is automatically a strict symmetric monoidal functor, so the adjective symmetric is not included in the above definition. 
\begin{prop}
There is an adjunction
\[
\begin{tikzcd}
\Petri \ar[r,bend left,"F"] \ar[r,phantom,"\bot"] & \CMC \ar[l,bend left,"U"] 
\end{tikzcd}
\]
whose left adjoint generates the free commutative monoidal category on a Petri net.
\end{prop}

\begin{proof}
    This is a special case of \cite[Thm.\ 5.1]{GeneralizedPetriNets} which shows that there is similar adjunction for any Lawvere theory $\mathsf{Q}$. When $\mathsf{Q}$ is the Lawvere theory for commutative monoids this theorem gives the desired adjunction. The above Theorem also contains a description of this left adjoint in more detail. 
\end{proof}

\noindent Note that many authors prefer to generate symmetric monoidal categories from Petri nets which are \emph{not} strictly commutative. Working out the details of this ask has many subtleties (e.g. in \cite{SassoneAxiom} to turn their categorical semantics functor into an adjunction they must quotient by natural transformations consisting only of symmetries). However, there are other options. In \cite{GeneralizedPetriNets}, we argue that Pre-nets, a non-commutative variant of Petri nets, more naturally generate non-commutative symmetric monoidal categories. In \cite{baez2021categories} we introduced $\Sigma$-nets, a Petri net variant where the user must specify which permutations alter a transition and which do not. This property makes them a middle ground betweeen Petri nets and Pre-nets: in a Pre-net every permutation alters its transition and in a Petri net none of them do. $\Sigma$-nets also admit a natural adjunction into strict symmetric monoidal categories which are not commutative.
\section{Colored Petri Nets}\label{sec:coloredpetri}

Jensen defined \emph{colored Petri nets} \cite{jensen}, and Lakos rephrased this definition as follows \cite{lakosabstraction}. Note that some sources refer to the following as \emph{colored Petri nets with guards}. 

\begin{defn}[\cite{lakosabstraction} Def.\ 4.8]\label{ogcnet}
    An (unmarked) \define{colored Petri net with guards} is a tuple \[(P, T, A, C, E)\] where
    \begin{itemize}
        \item $P$ is a set of \define{places}
        \item $T$ is a set of \define{transitions}
        \item $A \subseteq P \times T \cup T \times P$ is a set of \define{arcs}
        \item $C$ is an assignment of each place and transition to a color set i.e. a mapping 
        \[C \maps P + T \to \Set\]
        \item $E$ is the set of \define{arc inscriptions}. This is a set of functions
        \[E(a)\maps C(t) \to \N[C(p)] \]
        for every arc $a=(p, t)$ or $a=(t, p)$ in $A$
        \end{itemize}
\end{defn}
\noindent By this definition, a colored Petri net is a Petri net with a set of colors associated to each place, a set of colors associated to each transition, and functions between the free commutative monoids on each arc. A few remarks:

\begin{itemize} 
    \item the underlying Petri net $N=(P, T, A)$ of the above defintion can be interpreted as a Petri net in the sense of Definition \ref{def:Petri}. $N$ can be thought of as a pair of functions
    \[ \begin{tikzcd}T \ar[r, shift left=.5ex, "s"] \ar[r, shift right=.5ex, "t", swap] & \N[P] \end{tikzcd}\]
    where $s$ and $t$ send a transition $\tau$ to the sum of all places which are connected to $\tau$ by an input arc and output arc respectively.
    \item The guards are defined implicitly. Transitions which do not fire unless a condition is met are represented by arc inscriptions which produce colorings on the transitions which have no capability of leaving along an output arc inscription.
\end{itemize}
Colored Petri nets with guards can be interpreted as an assignment of syntax to semantics. The semantics will be the following double category which is a special case of the double category of spans $\Span(C)$ for a category $C$ with pullbacks (e.g. \cite[\S 3]{dawson2010span}).
\begin{prop}
Let $\Span(\CMon)$ denote the double category where 
\begin{itemize}
\item objects are sets,
\item horizontal 1-cells are spans in $\CMon$\[\begin{tikzcd}
    B &\ar[l]  A \ar[r] & C
\end{tikzcd}\]
\item horizontal composition is given by pullback in $\CMon$,
\item vertical 1-cells are monoid homomorphisms,
\item vertical 2-cells are commuting squares
\[\begin{tikzcd}
    B \ar[d,"f"] &\ar[l]  A\ar[d,"h"] \ar[r] &\ C \ar[d,"g"]\\
    B' & \ar[l] A' \ar[r] & C'
\end{tikzcd}\]
\item and vertical composition is given by composition of homomorphisms.
\end{itemize}
\end{prop}
\begin{prop}
$\Span(\CMon)$ may be equipped with the structure of a symmetric monoidal double category (e.g. \ref{defn:symmetric_monoidal_double_category}) $(
\Span(\CMon),\oplus)$ where
\begin{itemize}
    \item For objects $X,Y$ we have $X \oplus Y$ as their biproduct,
    \item For homomorphisms $f: X \to Y, g: X' \to Y'$, $f \oplus f' : X \oplus  X' \to Y \oplus Y'$ is their pairing.
    \item on horizontal 1-cells we have
    \[\begin{tikzcd} X & \ar[l,"l",swap]  A  \ar[r,"r"] & Y & \oplus & X' & \ar[l,"l'",swap]  A' \ar[r,"r'"] & Y'
\end{tikzcd}\]

\[ \begin{tikzcd}
{X \oplus X'}	& \ar[l,"{l\oplus l'}",swap] A \oplus A' \ar[r,"{r \oplus r'}"] & {Y \oplus Y'}
\end{tikzcd}\]
 \item and the action of plus on 2-cells is given by
 \[\begin{tikzcd}
	X & A & Y \\
	Q & B & R
	\arrow["{f}"', from=1-1, to=2-1]
	\arrow["l"', from=1-2, to=1-1]
	\arrow["r", from=1-2, to=1-3]
	\arrow["{s}", from=2-2, to=2-1]
	\arrow["{t}"', from=2-2, to=2-3]
	\arrow["{h}", from=1-3, to=2-3]
	\arrow["g", from=1-2, to=2-2]
\end{tikzcd} \oplus \begin{tikzcd}
	{X'} & {A'} & {Y'} \\
	{Q'} & {B'} & {R'}
	\arrow["{f'}"', from=1-1, to=2-1]
	\arrow["{l'}"', from=1-2, to=1-1]
	\arrow["{r'}", from=1-2, to=1-3]
	\arrow["{s'}", from=2-2, to=2-1]
	\arrow["{t'}"', from=2-2, to=2-3]
	\arrow["{h'}", from=1-3, to=2-3]
	\arrow["{g'}", from=1-2, to=2-2]
\end{tikzcd}\] 
\[=\begin{tikzcd}
	{X\oplus X'} && {A\oplus A'} && {Y\oplus Y'} \\
	{Q\oplus Q'} && {B\oplus B'} && {R\oplus R'}
	\arrow["{l \oplus l'}"', from=1-3, to=1-1]
	\arrow["{r \oplus r'}", from=1-3, to=1-5]
	\arrow["{s \oplus s'}", from=2-3, to=2-1]
	\arrow["{t \oplus t'}"', from=2-3, to=2-5]
	\arrow["{g\oplus g'}", from=1-3, to=2-3]
	\arrow["{h\oplus h'}", from=1-5, to=2-5]
	\arrow["{f\oplus f'}"', from=1-1, to=2-1]
\end{tikzcd}\]
\end{itemize}
\end{prop}
\begin{proof}
\cite[Ex.\,9.2]{FramedBicats} gives a symmetric cartesian monoidal structure $(\Span(C),\times)$ whenever $C$ has finite limits. We instantiate this in the case when $C = \CMon$. Note that $+$ is a biproduct on $\CMon$ which is in particular a cartesian product.
\end{proof}
\begin{prop}\label{correspond}
    A colored petri net with guards $(P, T, A, C, E)$ defines a symmetric monoidal strong double functor
    \[
        K \maps FN \to (\mathsf{Span} (\CMon), +) 
    \]
    where $FN$ is the free commutative monoidal category on the underlying Petri net \[N = \begin{tikzcd} T \ar[r, shift left = .5ex, "s"] \ar[r, shift right = .5ex, "t", swap] & \N[P] \end{tikzcd}.\]
\end{prop}

\begin{proof}
    Because $FN$ is freely generated both monoidally and compositionally, a functor out of $FN$ which preserves those two operations is uniquely defined by its assignment on generators. A place $p \in P$ is assigned to the free commutative monoid $\N[C(p)]$ and a transition $\tau \maps s(\tau) \to t(\tau) \in T$ is assigned to the span 
    \[\begin{tikzcd}
        &
        \N[C(\tau)]
        \ar[dl, "({E(x, t)})_{x \in s(\tau)}", swap] 
        \ar[dr, "({E(t, y)})_{y \in t(\tau)}"] 
        &\\
        \bigoplus_{x \in s(\tau)} \N[C(x)] 
        && 
        \bigoplus_{y \in t(\tau)} \N[C(y)]
    \end{tikzcd}\]
    Where $({E(t, y)})_{y \in t(\tau)}$ denotes the pairing of the arc inscriptions in the input of $\tau$ and $({E(t, y)})_{y \in t(\tau)}$ denotes the pairing of the arcs in the output of $\tau$. The rest of the objects and morphisms in $FN$ are freely generated by the places and transitions of $N$ so their assignment under $K$ is given by composites and direct sums of the above assignments. There is just one subtlety, $(FP, +)$ is strictly monoidal and commutative whereas $(\Span(\CMon), \oplus)$ only satisfies those properties up to coherent isomorphism. Luckily, the functor $K$ does not need to be strictly monoidal. This means that instead of having $K(a+b) = K(a) \oplus K(b)$, we have a coherent isomorphism \[ 
        \gamma_{a, b} \maps K(a) \oplus K(b) \cong K(a+b). 
    \]
    After choosing an initial association and permutation for the image of every sum in $FN$, this isomorphism will consist of whatever symmetries or associators in $\mathsf{Span}(\CMon)$ are necessary to preserve the commutativity and associativity equations in $FN$. Similarly, the composition comparison 
    \[\phi_{\tau,\sigma} : K (\tau) \circ K(\sigma) \cong K( \tau \circ \sigma)\]
    consists of the right symmetries and associators for pullback which turn the arbitary choices of associations and units into each other.
\end{proof}
\noindent To understand this characterization, it is helpful to look at an example.
\begin{examples}\label{colornet}
The following colored Petri net models the operation of a vending machine which dispense candy bars for $25 \cent$ and apples for $75 \cent$:
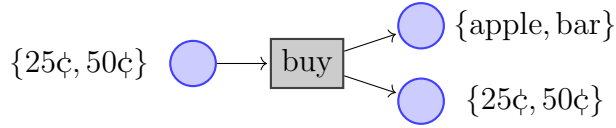
    \[
\begin{tikzpicture}
        \node (ca) at (0.5,0.5) {$\{\text{apple},\text{bar} \}$};
        \node (cb) at (0.5,-0.5) {$\{25\cent,50\cent\}$};
        \node (cc) at (-5.5,0) {$\{25 \cent,50\cent\}$};
        \node [style=place] (C) at (-1, .5) {};
        \node [style=place] (B) at (-1, -0.5) {};
        \node [style=place] (A) at (-4, 0) {};
        \node [style=transition] (tau1) at (-2.5, 0) {$\text{buy}$}
        edge [pre] (A)
        edge [post] (B)
        edge [post] (C);
\end{tikzpicture}\]
There are three species representing coins going in, food products going out, and coins going out. For each place, its set of colors is indicated next to it. The set of colors and their arc inscriptions are indicated as follows:
\begin{align*}
   25 \cent + 50 \cent & \leftmapsto  \text{a1}\mapsto (\text{apple},0)\\
  2 \cdot 50 \cent  & \leftmapsto \text{a2} \mapsto   (\text{apple},50\cent) \\
    3 \cdot 25 \cent  & \leftmapsto \text{a3} \mapsto   (\text{apple},0) \\
      25 \cent  & \leftmapsto \text{b1} \mapsto   (\text{bar},0) \\
    50 \cent  & \leftmapsto \text{b2} \mapsto   (\text{bar},25 \cent) \\
    25 \cent & \leftmapsto \text{e1} \mapsto (0,25 \cent)\\
        50 \cent & \leftmapsto \text{e2} \mapsto (0,50 \cent)
\end{align*}
The middle column gives the colors of the transition $\text{buy}$ and the columns to its left and right give the arc inscriptions for these colors on the input and output places respectively. Note that this table can be read as a span of free commutative monoids
\[\begin{tikzcd}\N[\{25 \cent, 50 \cent\} & \ar[l] \N[C
(\text{buy})] \ar[r] & \N[25 \cent, 50\cent] \oplus \N[\text{apple},\text{bar}] \end{tikzcd}\]
where $C(\text{buy})$ is the set whose elements are given by the middle column above. Let $P$ be the underlying Petri net.The above data defines a symmetric monoidal double functor $K \maps FP \to \Span(\CMon)$. The assignment on generators of $FP$ is shown above, and this assignment is freely extended to monoidal products and composites.
\end{examples}
One advantage of phrasing colored Petri nets in categorical terms is that it gives a natural definition of morphism between colored Petri nets:

\begin{defn}\label{morphism}
    Let $K \maps FP \to \Span(\CMon)$ and $K' \maps FP' \to \Span(\CMon)$ be colored Petri nets. Then a \define{morphism of colored Petri nets} $(f, g, \alpha)$ is a morphism of Petri nets $(f, g) \maps P \to P'$ and a monoidal vertical transformation $\alpha$ of the following shape.
    \[
    \begin{tikzcd}
        FP 
        \ar[dd, "{F(f, g)}", swap] 
        \ar[dr, "K"{name=U}, pos=0.2]
        &\\&  
        \Span(\CMon)
        \\
        FP'
        \ar[ur, "K'"{name=D}, swap, pos=0.15]
        & 
        \arrow[Rightarrow, from = U, to = D, shorten <=3ex, shorten >=3ex, "\alpha", swap]
    \end{tikzcd}
    \]
\end{defn}
\noindent Because this vertical transformation is monoidal, it suffices to define it on the generators of $FN$, i.e.\ the places and transitions of $N$ and not arbitrary sums of those places and transitions. Therefore, a vertical transformation as above consists of the following:
\begin{itemize}
    \item For each place $p$ a function between the color sets $\alpha_p \maps K(p) \to K'(g(p)) $. Because $FP$ is made into a double category trivially, the naturality square at this level is trivial.
    \item For each transition $\tau$, a morphism of spans i.e. a commuting diagram of the form
    \[
    \begin{tikzcd}
    K(s(\tau))\ar[d] & \ar[d]K(\tau) \ar[l] \ar[r] &\ar[d] K(t(\tau)) \\
    K'(s'(f(\tau)))) & K'(f(\tau)) \ar[l] \ar[r] & K'(t'(f(\tau)))
    \end{tikzcd}
    \]
    Because $FP$ has trivial vertical 2-cells, the naturality square at this level is also trivial.
    \item vertical transformations between double categories are required to respect composition and identities. Respecting identities is the condition that the morphism between identity spans
 \[
     \begin{tikzcd}
    K(x)\ar[d, "\alpha_x"] & \ar[d]K(x) \ar[l, equal] \ar[r, equal] &\ar[d, "\alpha_x"] K(x) \\
    K'(g(x)) & K'(g(x)) \ar[l, equal] \ar[r, equal] & K'(g(x))
    \end{tikzcd}
    \]
    has all three components given by $\alpha_x$. Respecting composition boils down to the requiring that $\alpha$ is built by extending its value on generators. For composable morphisms $\tau \maps x \to y$ and $\kappa \maps y \to z$ in $FN$, $\alpha$ gives the morphism of spans
    \[
    \begin{tikzcd}
    K(x)\ar[d, "\alpha_x"] & \ar[d, "\alpha_{K(\kappa \circ \tau)}"]K(\tau) \times_{K(y)} K(\kappa) \ar[l] \ar[r] &\ar[d, "\alpha_z"] K(z) \\
    K'(g(x)) & K'(f(\tau)) \times_{K'(g(y))} K'(f(\kappa)) \ar[l] \ar[r] & K'(g(z))
    \end{tikzcd}
    \]
    respecting composition requires that the arrow $\alpha_{K (\kappa \circ \tau))}$ is equal to the pullback of arrows $\alpha_{K(\tau)} \times_{\alpha_{K(y)}} \alpha_{K(\kappa)}$.
\end{itemize}

In \cite{lakosabstraction}, Lakos also defined morphisms of colored Petri nets. Roughly, a morphism from a colored Petri net $K=(P, T, A, C, E)$ to a colored Petri net $K'=(P', T', A', C', E')$ is a function between each set which can optionally satisfy the following properties:
\begin{itemize}
    \item {\bf Structure Respecting:} A morphism of colored Petri nets is structure respecting if it preserves arcs. This means that if an arc $a$ of $K$ connects a transition $t$ and a place $p$, then the image of this arc should connect the image of $t$ to the image of $p$. The above definition is structure respecting because it contains a morphism of the underlying Petri nets
    \item {\bf Color Respecting:} A morphism is color respecting if it sends the color on a place or a transition to a subset of the colors on the image of that place or transition. Our definition weakens this definition by requiring that there is an arbitrary monoid homomorphism between the two relevant color sets rather than an inclusion. A color respecting morphism must also preserve arc inscriptions. This is reflected in our definition by that fact that the components of $\alpha$ on the transitions are commuting diagrams of spans.
\end{itemize}
Two morphisms of colored Petri nets can be composed by composing each component separately. This forms the structure of a category.
\begin{defn}
Let \define{$\cPetri$} be the category of colored Petri nets and their morphisms.
\end{defn}

\section{Operational Semantics for Colored Petri Nets}\label{sec:unfolding}
Given a colored Petri net, there is a way to associate a category whose morphisms represent executions of your colored Petri net. To make this precise we will use the notion of step sequence. In \cite{lakosabstraction} these concepts are defined, here we adapt these definitions to our purposes:
\begin{defn}
Let $P$ be the petri net $\begin{tikzcd}T \ar[r,shift left=.5ex] \ar[r,shift right=.5ex] & \N[S] \end{tikzcd}$.
A \define{marking} of a colored Petri Net $K \maps FP \to \Span(\CMon)$ is a pair $(m,x)$ where $m$ is a marking of the Petri $P$ and $x$ is an element of $K(m)$. A \define{step} of $K$ with \define{source} $(m,x)$ and \define{target} $(n,y)$ is a tuple $(\tau,f)$ where
\begin{itemize}
    \item $\tau$ is an element of $\N[T]$ with source given by $m$ and target given by $n$ and,
    \item $f$ is an element of the apex of $K(\tau)$ whose image under the left and right legs of $K(\tau)$ are given by $x$ and $y$ respectively.
\end{itemize}
A \define{step sequence} is a sequence of composable steps $\{(\tau_i,f_i)\}_{i=1}^{k}$. Composable means that the target of $(\tau_i,f_i)$ is equal to the source of $(\tau_{i+1},f_{i+1})$ for all $1<i<k$. In particular, setting $k=0$ gives a step sequence from every marking to itself. The source of a step sequence is defined to be the source of its first step and the target of a step sequence is the target of its last step.
\end{defn}\noindent Our categorical operational semantics will use the Grothendieck construction for displayed categories. 

\subsection{Displayed Categories}

\begin{defn}
    A \define{displayed category} over $C$ is a lax double functor 
    \[ F \maps C \to \Span(\Set)\]
\end{defn}
\noindent Displayed categories have been studied by many authors. For example, in Jean B\'enabou's lecture notes \cite{benabounotes}. The term displayed categories was first coined in  \cite{displayed}. 
The construction of $\int D$ is a generalization of the Grothendieck construction \cite{SGAI}.

\begin{defn}\label{totalcat}
For a displayed category $F \maps C \to \Span(\Set)$ its total category $\int F$ is a category where:
\begin{itemize}
    \item An object is a tuple $(c,x)$ where $c$ is an object of $C$ and $x \in F(c)$.
    \item Let $f \maps c \to d$ be a morphism in $C$ with image
    \[\begin{tikzcd}
     & F(f)\ar[dr,"b"] \ar[dl,"a",swap]& \\
    F(c) & & F(d) \end{tikzcd} \]
    then for every object $r \in F(f)$, there is a morphism $(f,r) \maps (c,a(r)) \to (d,b(r))$.
    \item For morphisms $(f,r) \maps (c,x) \to (d,y)$ and $(g,s) \maps (d,y) \to (e,z)$ their composite is defined to be $(g \circ f, t) \maps (c,x) \to (e,z)$ where $t$ is the image of the pair $(r,s)$ under the laxator of $F$
    \[\phi : F(f) \times_{F(d)} F(g) \to F(g \circ f). \]
    \end{itemize}
    $\int F$ is equipped with a functor 
    \[\Int F \to C \]
    which sends every object and morphism to its first component.
\end{defn}
\noindent Displayed categories form a category and the total category construction is a functor.
\begin{defn}
Let $F \maps C \to \Span(\Set)$ and $G \maps D \to \Span(\Set)$ be displayed categories. Then a \define{morphism of displayed categories} $(j,\alpha) \maps F \to G$ is a functor $j \maps C \to D$  and a vertical transformation between double functors filling the following diagram
\[
    \begin{tikzcd}
        C
        \ar[dd,"{j}",swap] 
        \ar[dr,"F"{name=U},pos=0.2]
        &\\&  
        \Span(\Set)
        \\
    D
        \ar[ur,"G"{name=D}, swap, pos=0.2]
        & 
        \arrow[Rightarrow, from = U, to = D, shorten <=3ex, shorten >=3ex, "\alpha", swap]
    \end{tikzcd}
    \]
 This defines a category $\define{\mathsf{Disp}}$ with composition defined as composition of vertical transformations. Let $\Cat^2$ be the category where objects are functors $D \to C$ and morphisms are commuting squares
\[\begin{tikzcd}
D \ar[d] \ar[r] & D' \ar[d]\\
C \ar[r] & C'
\end{tikzcd}
\]

\end{defn}
\noindent The proof of the following theorem is a corollary of \cite[Thm. 3.45]{cruttwell2022double} when restricted to case where the vertical 1-cells are all identities (see also the discussion in \cite[pg. 472]{pare2011yoneda}). 
\begin{thm}
The total category construction gives an equivalence 
\[ \int \maps \mathsf{Disp} \xrightarrow{\sim} \Cat^2.\]
\end{thm}
\begin{prop}
    There is a symmetric monoidal double functor 
    \[ \U : \Span(\CMon) \to \Span\]
    given by extending the forgetful functor $U : \CMon \to \Set$ to the structure of $\Span(\CMon)$. In detail, 
    \begin{itemize}
        \item  a commutative monoid $X$ is sent to its underlying set $UX$,
        \item spans, vertical 1-cells, and 2-cells are sent to their pointwise application by $U$
    \end{itemize}
\end{prop}
\begin{proof}
The $\Span$ construction is a functor $\Span : \Cat \to \Dbl$ which we apply to the functor $U$.
\end{proof}
\noindent We may compose a colored Petri net $K : FN 
\to \Span(\CMon)$ with the double functor $\U$ to obtain a displayed category.
\begin{defn}\label{semantics}
The \define{categorical operational semantics functor} for colored Petri nets 
\[\int \maps \cPetri \to \Cat \]
is the composition of $\int \maps \Disp \to \Cat$ with $\Span(U) \maps \cPetri \to \Disp$.
Explicitly, for a colored Petri net $K \maps FP \to \Span(\CMon)$ to the category $\int K$ where:
\begin{itemize}
    \item An object is a pair $(m, x)$ where $m$ is an object of $FP$ and $x$ is an element of the commutative monoid $K(m)$.
    \item A morphism from $(m,x)$ to $(n,y)$ is a pair $(f,g)$ where $f \maps m \to n$ is a morphism in $FP$ and $g$ is an element in the apex of $F(f)$ which is mapped to $x$ and $y$ by the left and right legs of $F(f)$ respectively.
\end{itemize}
For a morphism of colored Petri nets 
   \[
    \begin{tikzcd}
        FP 
        \ar[dd, "{F(f, g)}", swap] 
        \ar[dr, "K"{name=U}, pos=0.2]
        &\\&  
        \Span(\CMon)
        \\
        FP'
        \ar[ur, "K'"{name=D}, swap, pos=0.15]
        & 
        \arrow[Rightarrow, from = U, to = D, shorten <=3ex, shorten >=3ex, "\alpha", swap]
    \end{tikzcd}
    \]
\end{defn}
\noindent The category $\int K$ is an operational semantics for the colored Petri net $K$ in the sense that morphisms of $\int K$ represent sequences of processes which can occur in $K$. 
\begin{prop}\label{soundness}
     There is bijection between step sequences in $K$ and morphisms in $\int K$. 
\end{prop}
\begin{proof}

Each step $(\tau ,f)$ with source $(m,x)$ and target $(n,y)$ gives a morphism 
$(\tau, f) \maps (m,x) \to (n,y)$. A step sequence $\{(\tau_i,f_i)\}_{i=1}^{k}$ gives a morphism by composing all the morphisms $(\tau_i,f_i)$ with each other in the natural way.

Conversely, let $(f,g) \maps (m,x) \to (n,y)$ be a morphism in $\int K$. $FP$ is freely generated by the transitions of $P$ under $+$ and composition. Therefore $f$ will be equal to an expression built from the basic transitions using $+$ and $\circ$.
The operation $+$ satisfies the interchange law
\[ c \circ a  + d \circ b =(c+ d) \circ (a +b) \]
whenever all composites are defined. This law can be used inductively to turn $f$ into a composite of sums
\[f_1 \circ f_2 \circ \ldots \circ f_n \]
where each $f_i \maps x_i \to x_{i+1}$ is a sum of transitions in $P$.
Because $K$ preserves composition strongly, there is an isomorphism
\[K(f) =K (f_1 \circ f_2 \ldots \circ f_n ) \cong K (f_1) \oplus_{K(x_2)} K(f_2) \oplus_{K(x_3)} \ldots \oplus_{K(x_{n}) } K(f_n).  \]
This isomorphism gives a way to decompose the element $g \in K(f)$ into a tuple of composable elements $g_1,g_2,\ldots, g_n$ where each $g_i$ is an element of $K(f_i)$. The pairs $(\tau_i,g_i)$ give a step sequence of $K$ which corresponds to the morphism $(f,g)$.
\end{proof}
\noindent Note that functoriality of $\int$ justifies the definition of morphism for colored Petri nets. Definition \ref{semantics} shows how morphisms of colored Petri nets lift to functors between their categories of step sequences. For a morphism of colored Petri nets $(j,\alpha) \maps K \to K'$, functoriality of $\int (j, \alpha) \maps \int K \to \int K'$ says that the mapping between step sequences induced by $\int (j,\alpha)$ respects concatenation.

\section{The Operational Semantics is Free on the Unfolding}\label{unfoldingtwo}Unfolding is a useful technique which reduces the analysis of colored Petri nets to that of ordinary Petri nets. This correspondence was first noticed by Jensen in the first volume of his book \emph{Coloured Petri nets: basic concepts, analysis methods and practical use} \cite{jensen2013coloured}. Since then, efficient algorithms have been developed for unfolding colored Petri nets \cite{systemsbio}, \cite{schwarick2020efficient}, \cite{kordon2006optimized}. It turns out that the operational semantics defined in the previous section is always the $\CMC$ on a some Petri net. The well-known unfolding construction gives this Petri net which generates $
\int K$.
\begin{defn}\label{unfold}
For a colored Petri net $(S,T,A,C,E)$ as in Definition \ref{ogcnet} with underlying Petri net $\begin{tikzcd}T \ar[r,shift left=.5ex,"s"] \ar[r,shift right=.5ex,"t",swap] & \N[S] \end{tikzcd}$. The \define{unfolding} of $K$ is the Petri net
\[
\begin{tikzcd} 
\UNFOLD(K) = \hat{T} \ar[r,shift left=.5ex,"\hat{s}"] \ar[r, shift right=.5ex,"\hat{t}",swap] & \N[\hat{S}]
\end{tikzcd}\]
where 
\begin{itemize}
    \item $\hat{T} = \{ (\tau, f) \, | \, \tau \in T \text{ and }f \in C(\tau)\}$,
    \item $\hat{S} = \{ (p, x) \,|\, p \in S \text{ and } x \in C(p)  \}$,
    \item $\hat{s} \maps \hat{T} \to \N[\hat{S}]$ sends a pair $(\tau,f)$ to the pair $(s(\tau), E((\tau,s(\tau)),f))$ and,
    \item $\hat{t} \maps \hat{T} \to \N[\hat{S}]$ sends a pair $(\tau, f)$ to the pair $(t(\tau),E((\tau,t(\tau)),f) )$.
\end{itemize}
\end{defn}
\noindent One advantage of reframing colored Petri nets in terms of category theory is that the definition of unfolding can be justified by categorical means.
The categorical semantics for Petri nets and colored Petri nets can be drawn together to get
\[
\begin{tikzcd} 
& \Petri \ar[dr,"F"] & \\
\cPetri \ar[rr,"\int",swap]  & &\CMC.
\end{tikzcd}
\]

\noindent To unfold a colored Petri net, we wish to find a functor $\mathtt{UNFOLD} \maps \cPetri \to \Petri$ such that for a colored Petri net $K$, the semantics of it's unfolding $F ( \mathtt{UNFOLD} (K) )$ matches it's original semantics $\int K$. This property says roughly that the behavior of an unfolded colored Petri must match the behavior of colored Petri net you started with. We may phrase this property as a triangle of functors commuting up to natural isomorphism.

\begin{thm}\label{unfold}
    The unfolding construction extends to a functor
    \[ \UNFOLD \maps \cPetri \to \Petri\]
    such that the diagram
\[
\begin{tikzcd} 
& \Petri \ar[dr,"F"] \ar[d,phantom,"\cong"] & \\
\cPetri \ar[ur,"\UNFOLD"] \ar[rr,"\int",swap]  & \, &\CMC.
\end{tikzcd}
\]
commutes up to natural isomorphism.
\end{thm}
\begin{proof}
$\UNFOLD$ is extended to morphisms as follows:
For a morphism of colored Petri nets $(f,g,\alpha) \maps K \to K'$, $\UNFOLD(f,g,\alpha) \maps \UNFOLD(K) \to \UNFOLD(K')$ consists of the maps
\begin{align*}
\hat{f}   \maps  \hat{T} \to \hat{T'} &\quad \quad\quad \text{and} &\hat{g}   \maps  \hat{S} \to \hat{S'}\\
      (\tau,k) \mapsto (f(\tau), \alpha_{\tau} (k) ) &  &(p,x) \mapsto (g(p), \alpha_{p} (x) )
\end{align*}
\noindent $(\hat{f},\hat{g})$ is indeed a morphism of Petri nets. Commuting with the source and target in the first component follows from $(f,g)$ being a morphism of Petri nets. Commuting with the source and target in the second component follows from $\alpha_\tau$ taking part in the commutative diagram
\[
\begin{tikzcd}
    K(s(\tau))
    \arrow[d, "\alpha_{s(\tau)}", swap]
    &
    K(\tau)
    \arrow[d, "\alpha_{\tau}"]
    \arrow[r]
    \arrow[l]
    &
    K(t(\tau))
    \arrow[d, "\alpha_{t(\tau)}"]
    \\
    K'(s'(f(\tau)))
    &
    K'(f(\tau))
    \arrow[r]
    \arrow[l]
    &
    K'(t'(f(\tau))).
\end{tikzcd}
\]
Functoriality of $\UNFOLD$ follows very quickly from the definitions.

Next we show that $\UNFOLD$ respects the categorical semantics functors for $\cPetri$ and $\Petri$. For a colored Petri net $K$, we have that an isomorphism of categories
\[F \circ \UNFOLD(K) \cong \int K\]
First note that we have a bijection
\[\hat{S} \cong \coprod_{p \in S} C(p) \]
between $\hat{S}$ and the disjoint union of all possible colors. Therefore, the objects of $F \circ \UNFOLD (K)$ are given by
\[\N[\hat{S}] \cong \N[\coprod_{p \in S} C(p)] \cong \oplus_{p \in S} \N[ C(p) ]  \]
where the last isomorphism follows from $\N \maps \Set \to \CMon$ preserving coproducts. On the other hand,
\[\Ob \int K = \{(m,x) \, | \, m \in \N[S] \text{ and } x \in K(m) \}. \]
Because $K$ preserves the monoidal product, $K(m)$ will be isomorphic to the direct sum of commutative monoids $K(p)$ for $p \in S$. Therefore, $m$ will be a tuple of markings $(m_1, m_2, \ldots, m_k)$ where each $m_i$ is an element of the free commutative monoid on the color set for some species. An isomorphism, $\Ob \int K \to \N[\hat{S}]$ of commutative monoids is constructed by sending the pair $(x,(m_1,m_2,\ldots,m_k))$ to the sum $\sum_{i} m_i$ which is evidently an element of $\N[\hat{S}]$. As in the proof of Proposition \ref{soundness}, we may use the interchange law to turn a morphism in $F \circ \UNFOLD (K)$ into a composite of tuples $(\tau_1,f_1)\circ (\tau_2,f_2)\circ \ldots \circ (\tau_n,f_n)$. This composite is mapped to the morphism of $\int K$ given by $(\tau_1 \circ \tau_2 \circ \ldots \circ \tau_n, (f_1,f_2,\ldots,f_n))$ which again by the interchange law represents every morphism in $\int K$.
\end{proof}

\begin{examples}
We can unfold the colored Petri net in Example \ref{colornet} to get the following:
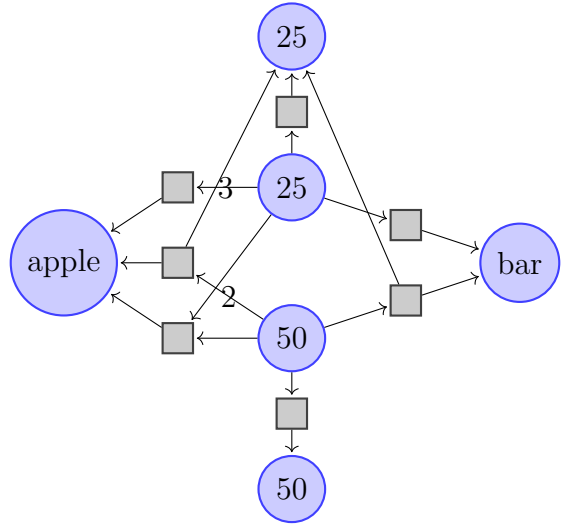
\[\begin{tikzpicture}
        \node [style=place] (25in) at (0,1) {25};
        \node [style=place] (50in) at (0,-1) {50};
        \node [style=place] (25out) at (0, 3) {25};
        \node [style=place] (50out) at (0,-3) {50};
        \node [style=place] (apple) at (-3,0) {apple};
        \node [style=place] (bar) at (3,0) {bar};
        \node [style=transition] (exit1) at (0, 2) {}
        edge [pre] (25in)
        edge [post] (25out);
        \node [style=transition] (exit2) at (0, -2) {}
        edge [pre] (50in)
        edge [post] (50out);
        \node [style=transition] (apple1) at (-1.5, 1) {}
        edge [pre] node {3} (25in)
        edge [post] (apple);
        \node [style=transition] (apple2) at (-1.5, -1) {}
        edge [pre]  (25in)
        edge [pre]  (50in)
        edge [post] (apple);
        \node [style=transition] (apple3) at (-1.5, 0) {}
        edge [pre] node {2} (50in)
        edge [post] (apple)
        edge [post] (25out);
        \node [style=transition] (bar1) at (1.5, .5) {}
        edge [pre] (25in)
        edge [post] (bar);
        \node [style=transition] (bar2) at (1.5, -.5) {}
        edge [pre] (50in)
        edge [post] (bar)
        edge [post] (25out);
\end{tikzpicture}\]
The reader may verify that there is indeed an isomorphism $\int K \cong F \circ \UNFOLD K$.
\end{examples}

\section{Conclusion}
Characterizing colored Petri nets categorically will allow them to be fit into the larger program of applied category theory. In particular, we conjecture:
\begin{itemize}
\item Using the structured cospan construction \cite{StructuredCospans}, we can create a compositional framework for open colored Petri nets. These would be colored Petri nets which are equipped with a boundary on their inputs and outputs. One advantage of this would be a categorical description of the reachability problem for colored Petri nets as shown in \cite{OpenPetriNets}.
\item We may be able to construct a \emph{double operadic systems theory} as in \cite{libkind2025doubleoperadictheorysystems} where open colored Petri nets are the objects. We believe that the double categorical presentation of colored Petri nets given in this paper is more ammenable to fitting into the above sort of systems theory.
\item We aim to find a full categorical framework for SDCPNs (as first defined in \cite{everdij1978compositional}). Since SDCPNs are colored Petri nets equipped with lots of bells and whistles, this paper represents a first step in the above goal.
\item The displayed category construction as stated in this paper ignores the monoidal structure which is present in colored Petri nets. In \cite{moeller2020monoidal}, the Grothendieck construction is lifted to an equivalence between monoidal indexed categories and monoidal fibrations. We believe that the displayed category construction could be similarly lifted to an equivalence between \emph{monoidal} displayed categories and \emph{monoidal} functors. We aim to develop the theory of monoidal displayed categories in order to further refine our understanding of the categorical semantics of colored Petri nets.
\end{itemize}
\appendix
\section{Double categories}

\subsection*{Double Categories}

What follows are brief definitions of double categories, lax and strong double functors, vertical transformations, monoidal double categories, and lax monoidal double functors. A more detailed exposition can be found in the work of Grandis and Par\'e \cite{grandis2004adjoint,grandis1999limits}, and for monoidal double categories the work of Shulman \cite{FramedBicats}. We use `double category' to mean what earlier authors called a `pseudo double category'.

\begin{defn}
\label{defn:double_category}
    A \textbf{double category} is a category weakly internal to $\Cat$. More explicitly, a double category $\mathbb{D}$ consists of:
    \begin{itemize}
        \item a \define{category of objects} $\mathbb{D}_0$ and a \define{category of arrows} $\mathbb{D}_1$,
        \item  \define{source} and \define{target} functors
        \[  
            S,T \colon \mathbb{D}_1 \to \mathbb{D}_0 ,
        \]
        an \define{identity-assigning} functor
        \[  
            U\colon \mathbb{D}_0 \to \mathbb{D}_1 ,
        \]
        and a \define{composition} functor
        \[ 
            \odot \colon \mathbb{D}_1 \times_{\mathbb{D}_0} \mathbb{D}_1 \to \mathbb{D}_1 
        \]
        where the pullback is taken over $\mathbb{D}_1 \xrightarrow[]{T} \mathbb{D}_0 \xleftarrow[]{S} \mathbb{D}_1$,
        such that
        \[  
            S(U_{A})=A=T(U_{A}) , \quad
        	S(M \odot N)=SN, \quad
            T(M \odot N)=TM, 
        \]
        \item natural isomorphisms called the \define{associator}
        \[ 
            \alpha_{N,N',N''} \maps (N \odot N') \odot N'' \xrightarrow{\sim} N \odot (N' \odot N'') , 
        \]
        the \define{left unitor}
        \[		
            \lambda_N \maps U_{T(N)} \odot N \xrightarrow{\sim} N, 
        \]
        and the \define{right unitor}
        \[  
            \rho_N \maps N \odot U_{S(N)} \xrightarrow{\sim} N 
        \]
        such that $S(\alpha), S(\lambda), S(\rho), T(\alpha), T(\lambda)$ and $T(\rho)$ are all identities and such that the standard coherence axioms hold: the pentagon identity for the associator and the triangle identity for the left and right unitor \cite[Sec.\ VII.1]{mac1998categories}.
    \end{itemize}
    If $\alpha$, $\lambda$ and $\rho$ are identities, we call $\mathbb{D}$ a \define{strict} double category.
\end{defn}

Objects of $\mathbb{D}_0$ are called \define{objects} and morphisms in $\mathbb{D}_0$ are called \define{vertical 1-morphisms}.  Objects of $\mathbb{D}_1$ are called \define{horizontal 1-cells} of $\mathbb{D}$ and morphisms in $\mathbb{D}_1$ are called \define{2-morphisms}.   A morphism $\alpha \maps M \to N$ in $\mathbb{D}_1$ can be drawn as a square:
\[
\begin{tikzpicture}[scale=1]
\node (D) at (-4,0.5) {$A$};
\node (E) at (-2,0.5) {$B$};
\node (F) at (-4,-1) {$C$};
\node (A) at (-2,-1) {$D$};
\node (B) at (-3,-0.25) {$\Downarrow \alpha$};
\path[->,font=\scriptsize,>=angle 90]
(D) edge node [above]{$M$}(E)
(E) edge node [right]{$g$}(A)
(D) edge node [left]{$f$}(F)
(F) edge node [above]{$N$} (A);
\end{tikzpicture}
\]
where $f = S\alpha$ and $g = T\alpha$. Note that every category $C$ can be upgraded to a double category: the sets $\Ob \,\,C$ and $\Mor \, \,C$ can be regarded as categories with only identity morphisms and the structure maps (source, target, identities, and composition) extend to functors between these categories in a unique way. In this way $C$ gives a trivial double category where the only vertical morphisms and $2$-morphisms are given by identities. 

There are maps between double categories, and also transformations between maps:
\begin{defn}
\label{defn:double_functor}
Let $\mathbb{A}$ and $\mathbb{B}$ be double categories. A \textbf{double functor} $F \maps \mathbb{A} \to \mathbb{B}$ consists of:
\begin{itemize}
\item functors $F_0 \maps \mathbb{A}_0 \to \mathbb{B}_0$ and $F_1 \maps \mathbb{A}_1 \to \mathbb{B}_1$ obeying the following
equations: 
\[S \circ F_1 = F_0 \circ S, \qquad T \circ F_1 = F_0 \circ T,\]
\item natural isomorphisms called the \define{composition comparison}: 
\[   \phi(N,N') \maps F_1(N) \odot F_1(N') \stackrel{\sim}{\longrightarrow} F_1(N \odot N') \]
and the \define{identity comparison}:
\[  \phi_{A} \maps U_{F_0 (A)} \stackrel{\sim}{\longrightarrow} F_1(U_A) \]
whose components are globular 2-morphisms, 
\end{itemize}
such that the following diagram commmute:
\begin{itemize} 
\item a diagram expressing compatibility with the associator: 
\[\xymatrix{ 	(F_1(N) \odot F_1(N')) \odot F_1(N'') \ar[d]_{\phi (N,N') \odot 1} \ar[r]^{\alpha} & F_1(N) \odot (F_1(N') \odot F_1(N'')) \ar[d]^{1 \odot \phi(N',N'')} \\
			F_1(N \odot N') \odot F_1(N'') \ar[d]_{\phi(N \odot N', N'')} & F_1(N) \odot F_1(N' \odot N'') \ar[d]^{\phi(N, N'\odot N'')}\\
F_1((N \odot N') \odot N'') \ar[r]^{F_1(\alpha)} & F_1(N \odot (N' \odot N'')) }	\]
\item two diagrams expressing compatibility with the left and right unitors:
	\[
	\begin{tikzpicture}[scale=1.5]
	\node (A) at (1,1) {$F_1(N) \odot U_{F_0(A)}$};
	\node (A') at (1,0) {$F_1(N) \odot F_1(U_{A})$};
	\node (C) at (3.5,1) {$F_1(N)$};
	\node (C') at (3.5,0) {$F_1(N \odot U_A)$};
	\path[->,font=\scriptsize,>=angle 90]
	(A) edge node[left]{$1 \odot \phi_{A}$} (A')
	(C') edge node[right]{$F_1(\rho_N)$} (C)
	(A) edge node[above]{$\rho_{F_1(N)}$} (C)
	(A') edge node[above]{$\phi(N,U_{A})$} (C');
	\end{tikzpicture}
	\]
	\[
	\begin{tikzpicture}[scale=1.5]
	\node (B) at (5.5,1) {$U_{F_0(B)} \odot F_1(N)$};
	\node (B') at (5.5,0) {$F_1(U_{B}) \odot F_1(N)$};
	\node (D) at (8,1) {$F_1(N)$};
	\node (D') at (8,0) {$F_1(U_{B} \odot N).$};
		\path[->,font=\scriptsize,>=angle 90]
		(B) edge node[left]{$\phi_{B} \odot 1$} (B')
	(B') edge node[above]{$\phi(U_{B},N)$} (D')
	(B) edge node[above]{$\lambda_{F_1(N)}$} (D)
	(D') edge node[right]{$F_1(\lambda_{N})$} (D);
	\end{tikzpicture}
	\]
\end{itemize}
If the 2-morphisms $\phi(N,N')$ and $\phi_A$ are identities for all $N,N' \in \mathbb{A}_1$ and 
$A \in \mathbb{A}_0$, we say $F \maps \mathbb{A} \to \mathbb{B}$ is a \define{strict} double functor.  If on the other hand we drop the requirement that these 2-morphisms be invertible, we call $F$ a \define{lax} double
functor.
\end{defn}
	
\begin{defn}
Let $F \maps \mathbb{A} \to \mathbb{B}$ and $G \maps \mathbb{A} \to \mathbb{B}$ be lax double functors. A \define{vertical transformation} $\beta \maps F \Rightarrow G$ consists of natural transformations $\beta_0 \maps F_0 \Rightarrow G_0$ and $\beta_1 \maps F_1 \Rightarrow G_1$ (both usually written as $\beta$) such that 
		\begin{itemize}
			\item $S( \beta_M) = \beta_{SM}$ and $T(\beta_M) = \beta_{TM}$ for any object $M \in A_1$, 
			\item $\beta$ commutes with the composition comparison, and
			\item $\beta$ commutes with the identity comparison.
		\end{itemize}
\end{defn}
	
Shulman defines a 2-category $\mathbf{Dbl}$ of double categories, double functors, and vertical transformations \cite{FramedBicats}. This has finite products.  In any 2-category with finite products we can define a pseudomonoid \cite{Monoidalbicatshopfalgebroids}, which is a categorification of the concept of monoid.  For example, a pseudomonoid in $\mathsf{Cat}$ is a monoidal category.
	
\begin{defn}
\label{defn:monoidal_double_category}
    A \textbf{monoidal double category} is a pseudomonoid in $\mathbf{Dbl}$. Explicitly, a monoidal double category is a double category equipped with double functors $\otimes \maps \mathbb{D} \times \mathbb{D} \to \mathbb{D}$ and $I \maps * \to \mathbb{D}$ where $*$ is the terminal double category, along with invertible vertical transformations called the \define{associator}:
    \[  
        A \maps \otimes \, \circ \; (1_{\mathbb{D}} \times \otimes ) \Rightarrow \otimes \; \circ \; (\otimes \times 1_{\mathbb{D}}) ,
    \]
\define{left unitor}:
\[ L\maps \otimes \, \circ \; (1_{\mathbb{D}} \times I) \Rightarrow 1_{\mathbb{D}} ,\]
and \define{right unitor}:
\[ R \maps \otimes \,\circ\; (I \times 1_{\mathbb{D}}) \Rightarrow 1_{\mathbb{D}} \]
satisfying the pentagon axiom and triangle axioms.
\end{defn}

This definition neatly packages a large quantity of information.   Namely:
\begin{itemize}
\item $\mathbb{D}_0$ and $\mathbb{D}_1$ are both monoidal categories.
\item If $I$ is the monoidal unit of $\mathbb{D}_0$, then $U_I$ is the
monoidal unit of $\mathbb{D}_1$.
\item The functors $S$ and $T$ are strict monoidal.
\item $\otimes$ is equipped with composition and identity comparisons
\[ \chi \maps (M_1\otimes N_1)\odot (M_2\otimes N_2) \stackrel{\sim}{\longrightarrow}
(M_1\odot M_2)\otimes (N_1\odot N_2)\]
\[ \mu \maps U_{A\otimes B} \stackrel{\sim}{\longrightarrow} (U_A \otimes U_B)\]
making three diagrams commute as in the definition of double functor.

\item The associativity isomorphism for $\otimes$ is a vertical transformation between double functors.
		\item The unit isomorphisms are vertical transformations
between double functors.
	\end{itemize}

	\begin{defn}
	\label{defn:symmetric_monoidal_double_category}
A \define{braided monoidal double category} is a monoidal double
category equipped with an invertible vertical transformation
\[ \beta \maps \otimes \Rightarrow \otimes \circ \tau \]
called the \define{braiding}, where $\tau \maps \mathbb{D} \times \mathbb{D} \to \mathbb{D} \times \mathbb{D}$ is the twist double functor sending pairs in the object and arrow categories to the same pairs in the opposite order. The braiding is required to satisfy the usual two hexagon identities \cite[Sec.\ XI.1]{mac1998categories}.  If the braiding is self-inverse we say that $\mathbb{D}$ is a \define{symmetric monoidal double category}.
	\end{defn}
	
In other words:
\begin{itemize}
		\item $\mathbb{D}_0$ and $\mathbb{D}_1$ are braided (resp. symmetric) monoidal categories,
		\item the functors $S$ and $T$ are strict braided monoidal functors, and
		\item the braiding is a vertical transformation between double functors.
\end{itemize}

\begin{defn}
\label{defn:monoidal_double_functor}
    A \define{monoidal lax double functor} $F \colon \mathbb{C} \to \mathbb{D}$ between monoidal double categories $\mathbb{C}$ and $\mathbb{D}$ is a lax double functor $F \maps \mathbb{C} \to \mathbb{D}$ such that
	\begin{itemize}
		\item $F_0$ and $F_1$ are monoidal functors,
		\item $SF_1= F_0S$ and $TF_1 = F_0T$ are equations between monoidal functors, and
		\item the composition and unit comparisons $\phi(N_1,N_2) \maps F_1(N_1) \odot F_1(N_2) \to F_1(N_1\odot N_2)$ and $\phi_A \maps U_{F_0 (A)} \to F_1(U_A)$ are monoidal natural vertical transformations.
	\end{itemize}
    The monoidal lax double functor is \define{braided} if $F_0$ and $F_1$ are braided monoidal functors and \define{symmetric} if they are symmetric monoidal functors. 
    
    A fundamental example is the symmetric monoidal double category of spans.
    \begin{defn}\label{span}
    Let $\Span$ denote the double category where
\begin{itemize}
\item the category of objects is given by $\Set$: the category of sets and functions,
\item the category of morphisms $\Span_1$ has objects given by spans of functions
   \[\begin{tikzcd}A &\ar[l]X \ar[r] & B \end{tikzcd}\]
and a morphism from $\begin{tikzcd}A &\ar[l]X \ar[r] & B \end{tikzcd}$ to $\begin{tikzcd}A' &\ar[l]Y \ar[r] & B' \end{tikzcd}$ is a commuting diagram
    \[
    \begin{tikzcd}
        A\ar[d] & X \ar[r] \ar[d] \ar[l] & B \ar[d] \\
        A' & Y \ar[r] \ar[l] & B'.
    \end{tikzcd}
    \]
    \item The source and target functors $S,T \maps \Span_1 \to \Set $ send a span \[\begin{tikzcd}A &\ar[l]X \ar[r] & B \end{tikzcd}\] to $A$ and $B$ respectively.
    \item The identity assigning functor $U\maps \Set \to \Span_1$ sends a set $X$ to the span \[\begin{tikzcd}X &\ar[l,"1_x",swap]X \ar[r,"1_x"] & X \end{tikzcd}\].
    \item The composition functor, $\circ \maps \Span_1 \times_{\Set} \Span_1 \to \Span_1$, sends a pair of spans to their pullback and a pair of 2-morphisms to the universal map induced by the pullback.
    \item The associator, the left unitor and, the right unitor are all given by canonical isomorphisms.
\end{itemize} 
\end{defn}
\end{defn}

\bibliographystyle{alpha}
\bibliography{references}


\end{document}